\documentclass[11pt]{article}
\usepackage{latexsym,amsmath,amssymb,amsthm}
\usepackage{setspace}
\usepackage{amssymb, amsmath, amsthm, graphicx,mathrsfs}
\usepackage[left=1in,top=1in,right=1in,bottom=1in]{geometry}
\usepackage{subfigure}
\usepackage{titlesec}
\titleformat{\subsection}[hang]
  {\normalfont\bfseries}{\thesubsection}{1em}{}

      \newtheorem{theorem}{Theorem}
      \newtheorem{lemma}[theorem]{Lemma}
      \newtheorem{problem}[theorem]{Problem}
      \newtheorem{corollary}[theorem]{Corollary}
      
      \newtheorem{proposition}[theorem]{Proposition}
      \newtheorem{conjecture}[theorem]{Conjecture}
      \newtheorem{definition}[theorem]{Definition}

\def\ex{{\rm{ex}}}

\def\TT{{\mathcal T}}
\def\Tab{{Let $T$ be a tree on $s+t$ vertices and let ${\mathcal T}=T(a,b)$, its $(a,b)$-blowup. }}

\begin{document}

\pagestyle{myheadings}
\markright{{\small{\sc F\"uredi, Jiang, Kostochka, Mubayi, and Verstra\"ete:   Hypergraph blowups of trees}}}

\title{\vspace{-0.5in} Extremal problems for hypergraph blowups of trees\footnote{This
project started at SQuaRE workshop at the American Institute of Mathematics.}}

\author{
{\large{Zolt\'an F\"uredi}}\thanks{
\footnotesize {Alfr\'ed R\'enyi Institute of Mathematics, Budapest, Hungary.
E-mail:  \texttt{z-furedi@illinois.edu}.
Research  partially supported by grant KH130371 from the National Research, Development and Innovation Office NKFIH.}}
\and
{\large{Tao Jiang}}\thanks{
\footnotesize {Department of Mathematics, Miami University, Oxford, OH 45056,
E-mail:  \texttt{jiangt@miamioh.edu}.
Research partially supported by National Science Foundation award DMS-1400249.}}
\and
{\large{Alexandr Kostochka}}\thanks{
\footnotesize {University of Illinois at Urbana--Champaign, Urbana, IL 61801
 and Sobolev Institute of Mathematics, Novosibirsk 630090, Russia. E-mail: \texttt {kostochk@math.uiuc.edu}.
 Research supported in part by NSF grant  DMS-1600592 and  by grants 18-01-00353A and 19-01-00682
 of the Russian Foundation for Basic Research.
}}
\and
{\large{Dhruv Mubayi}}\thanks{
\footnotesize {Department of Mathematics, Statistics, and Computer Science, University of Illinois at Chicago, Chicago, IL 60607.
E-mail:  \texttt{mubayi@uic.edu.}
Research partially supported by NSF awards DMS-1300138 and 1763317.}}
\and{\large{Jacques Verstra\"ete}}\thanks{Department of Mathematics, University of California at San Diego, 9500
Gilman Drive, La Jolla, California 92093-0112, USA. E-mail: {\tt jverstra@math.ucsd.edu.} Research supported by NSF awards DMS-1556524
and DMS-1800332.}}

\date{March 2, 2020}
\maketitle
\vspace{-0.3in}

\begin{abstract}

In this paper we present a novel approach in extremal set theory which may be viewed as an asymmetric version of Katona's permutation method. We use it to find more Tur\'an numbers of hypergraphs in the Erd\H{o}s--Ko--Rado range.

An $(a,b)$-path $P$ of length $2k-1$ consists of $2k-1$ sets of size $r=a+b$ as follows.
Take $k$ pairwise disjoint $a$-element sets $A_0, A_2, \dots, A_{2k-2}$ and other $k$ pairwise disjoint $b$-element sets
 $B_1, B_3, \dots, B_{2k-1}$ and order them linearly as $A_0, B_1, A_2, B_3, A_4\dots$.
Define the (hyper)edges of $P_{2k-1}(a,b)$ as the sets of the form $A_i\cup B_{i+1}$ and $B_j\cup A_{j+1}$.
The members of $P$ can be represented as $r$-element intervals of the $ak+bk$ element underlying set.

Our main result is about hypergraphs that are blowups of trees, and implies that for fixed $k,a,b$, as $n\to \infty$
\[ \ex_r(n,P_{2k-1}(a,b)) = (k - 1){n \choose r - 1} +
o(n^{r - 1}).\]

This generalizes the Erd\H os--Gallai theorem for graphs which is the case of $a=b=1$.
We also determine the asymptotics when $a+b$ is even; the remaining cases are still open.
\end{abstract}

\section{Paths}

\subsection{Definitions concerning $r$-uniform hypergraphs, Two constructions}
An $r$-uniform hypergraph, or simply {\em $r$-graph}, is a family of $r$-element subsets of a finite set.
We associate an $r$-graph $F$ with its edge set and call its vertex set $V(F)$.
Usually we take $V(F)=[n]$, where $[n]$ is the set of first $n$ integers, $[n]:=\{ 1, 2, 3,\dots, n\}$.
We also use the notation $F\subseteq \binom{[n]}{r}$.
For a hypergraph $H$, a vertex subset $C$ of $H$ that intersects all edges of $H$ is called
a {\em vertex cover} of $H$.
Let $\tau(H)$ be the minimum size of a vertex cover of $H$.
Let $\Psi_c(n,r)$ be the $r$-graph with vertex set $[n]$ consisting of all $r$-edges meeting $[c]$.
Then $\Psi$ has the maximum number of $r$-sets such that $\tau(\Psi)\leq c$.
When $r$ and $c$ are fixed and $n\to \infty$,
\begin{equation}\label{eq:psi}
  |\Psi_c(n,r)|=   \binom{n}{r}- \binom{n-c}{r}    =\, c{n \choose r - 1} + o(n^{r - 1}).
  \end{equation}

A {\em crosscut} of a hypergraph $H$ is a set $X \subset V(H)$ such that $|e \cap X| = 1$ for all $e \in H$.
Not all hypergraphs have crosscuts.
Let $\sigma(H)$ denote the smallest size of a crosscut in
a hypergraph $H$ with at least one crosscut.
Clearly $\tau(H)\le \sigma(H)$, since a crosscut is
a vertex cover. Here strict inequality may hold, as shown by a double star whose adjacent
centers have high degrees.
Define $ \Psi^1_{c}(n,r):=\{ E\subset [n] : |E|=r,  |E  \cap [c] |= 1\}$, so it
consists of all $r$-sets intersecting a fixed $c$-element subset of $V(H)$ at {\em exactly} one vertex.
Then for large enough $n$, $\Psi^1$ has the maximum number of $r$-sets such that $\sigma(\Psi^1)\leq c$.
Let us refer to this hypergraph as the {\em crosscut construction}.
When $r$ and $c$ are fixed and $n\to \infty$,
\begin{equation}\label{eq:psi1}
  |\Psi^1_c(n,r)|=   c \binom{n-c}{r-1} =\, c{n \choose r - 1} + o(n^{r - 1}).
  \end{equation}

Given an $r$-graph $F$,
let $\ex_r(n,F)$ denote the maximum number of edges in an $r$-graph on $n$ vertices that does not contain a copy of $F$ (if the uniformity is obvious from context, we may omit the subscript $r$).
Crosscuts were introduced in~\cite{FF} to get the following obvious lower bounds
\begin{equation}\label{eq3}
\ex(n,F) \geq |\Psi_{\tau(F)-1}(n,r)|,\quad \text{and if crosscut exists then }\enskip  \ex(n,F)
\geq |\Psi^1_{\sigma(F)-1}(n,r)|.
\end{equation}

\medskip

{\bf Notation.} If $H$ is a hypergraph and $e \subset
V(H)$, the {\em neighborhood} of $e$ is $\Gamma_H(e) =
\{f\setminus e : e \subseteq f$, $f \in H\}$ and the {\em degree} of $e$ is
$d_H(e) = |\Gamma_H(e)|$. For an integer $ p$, let  the {\em $p$-shadow}, $\partial_p
H$, be the collection of $p$-sets that lie in some edge of $H$. If $H$ is an $r$-graph, then the
$(r-1)$-shadow of $H$ is simply called the {\em shadow} and is denoted by  $\partial H$.

Whenever we write $f(n)\sim g(n)$, we always mean $\lim _{n \to \infty }f(n)/g(n)=1$ while the other variables of $f$ and $g$
are fixed. This is the case even if the variable $n$ is not indicated.

\medskip

{\bf Aims of this paper.}
We have two aims.
First,  to find more Tur\'an numbers (or estimates) of  hypergraphs in the Erd\H{o}s--Ko--Rado range.
We are especially interested in cases when the excluded configuration is 'dense', it has only a few vertices of degree one.
Second, we present an asymmetric version of Katona's permutation method, when we first solve (estimate) the problem only on a wellchosen substructure.
The $(a,b)$-blowups of trees and paths  are good examples for both of our aims.

\subsection{Paths in graphs}
A fundamental result in extremal graph theory is the
Erd\H os--Gallai Theorem~\cite{EG}, that
\begin{equation}\label{eq:2}
 \ex_2(n,P_\ell) \leq \frac{1}{2}(\ell - 1)n, \end{equation}
where $P_\ell$ is the $\ell$-{\it edge path}. (Warning! This is a non-standard notation).
Equality holds in~\eqref{eq:2} if and only if $\ell$ divides $n$ and all connected
 components of $G$ are $\ell$-vertex complete graphs.
The Tur\'an function $\ex(n,P_\ell)$ was determined exactly for every $\ell$ and $n$ by
Faudree and Schelp~\cite{FaudreeSchelpJCT75} and independently by Kopylov~\cite{Kopylov}.
Let $n\equiv r$ {\rm (mod $\ell$)}, $0\le r< \ell$. 
Then
$ \ex(n ,P_\ell)= \frac{1}{2} (\ell-1)n - \frac{1}{2}r(\ell-r)$.
They  also described the extremal graphs which are either
\hfill\break\phantom{~}\hskip\parindent
  --- vertex disjoint unions of $\lfloor n/\ell \rfloor$ complete graphs $K_{\ell}$ and a $K_r$, or
\hfill\break\phantom{~}\hskip\parindent
  --- $\ell$ is odd,  $\ell= 2k-1$, and $r=k$ or $k-1$. Then other extremal graphs with completely different structure can be obtained by taking a vertex disjoint union of $m$ copies of $K_{\ell}$ ($0\le m < \lfloor n/\ell\rfloor $) and a copy of $\Psi_{k-1}(n-m\ell, 2)$, i.e., an $(n-m\ell)$-vertex graph with a $(k-1)$-set meeting all edges.

This variety of extremal graphs makes the solution difficult.

We generalize these theorems for some hypergraph paths and trees.

\subsection{Paths in hypergraphs}

{\bf Paths of length $2$.} \quad
Two $r$-sets with intersection size $b$ can be considered as a hypergraph path $P_2(a,b)$ of length two, where $a+b=r$, and $1\leq a, b\leq r-1$.
If $H\subset \binom{[n]}{r}$ is $P_2(1,r-1)$-free then the obvious inequality $r |H|=|\partial(H)|\leq \binom{n}{r-1}$ yields the upper bound in the following result:
\begin{equation}\label{eq:3}
   \frac{1}{r}\binom{n}{r-1}-O(n^{r-2})< P(n, r, r-1)=\ex_r(n, P_2(1,r-1)) \leq  \frac{1}{r}\binom{n}{r-1}.
    \end{equation}
Here for any given $r$ equality holds if $n$ is sufficiently large ($n > n_0(r)$) and certain divisibility conditions are satisfied (see,
Keevash~\cite{Keevash}).

The case $b=1$ was solved asymptotically by Frankl~\cite{Frankl1977} and the general case was handled in~\cite{FF85}. 
\begin{equation}\label{eq:4}
   \ex_r(n, P_2(a,b)) = \Theta\left( n^{\max\{ a-1,b \}}\right). 
   \end{equation} 
Here the right hand side of~\eqref{eq:4} is $o(n^{r-1})$ (for $1\leq a, b\leq r-1$). 

Two disjoint $r$-sets can be considered as a $P_2(r,0)$ so~\eqref{eq:4} also holds for $a=r$ since
 the maximum size of an intersecting family of $r$-sets is $\binom{n-1}{r-1}$ for $n\geq 2r$
 by the Erd\H{o}s-Ko-Rado theorem~\cite{EKR}.

\medskip
{\bf Definition.}\quad
Suppose that $a,b,\ell$ are positive integers, $r=a+b$.
The  $(a,b)$-{\em path} $P_{\ell}(a,b)$ {\em of length} $\ell$ is an $r$-uniform hypergraph obtained from a (graph) path $P_\ell$ by blowing up its vertices to  $a$-sets and $b$-sets. More precisely,
  an $(a,b)$-path $P_{\ell}(a,b)$ of length $2k-1$ consists of $2k-1$ sets of size $r=a+b$ as follows.
Take $2k$ pairwise disjoint sets  $A_0, A_2, \dots, A_{2k-1}$ with $|A_i|=a$ and
 $B_1, B_3, \dots, B_{2k-1}$ with $|B_j|=b$ and define the (hyper)edges of $P_{2k-1}(a,b)$
as the sets of the form $A_i\cup B_{i+1}$ and $B_j\cup A_{j+1}$.
If the $ak+bk$ elements are ordered linearly, then the members of $P$ can be represented as intervals of length $r$.
By adding one more set $A_{2k}$ to the underlying set together with the hyperedge $B_{2k-1}\cup A_{2k}$ we obtain the
 $(a,b)$-path of even length, $P_{2k}(a,b)$.

\medskip\begin{center}
    \includegraphics[scale=0.3]{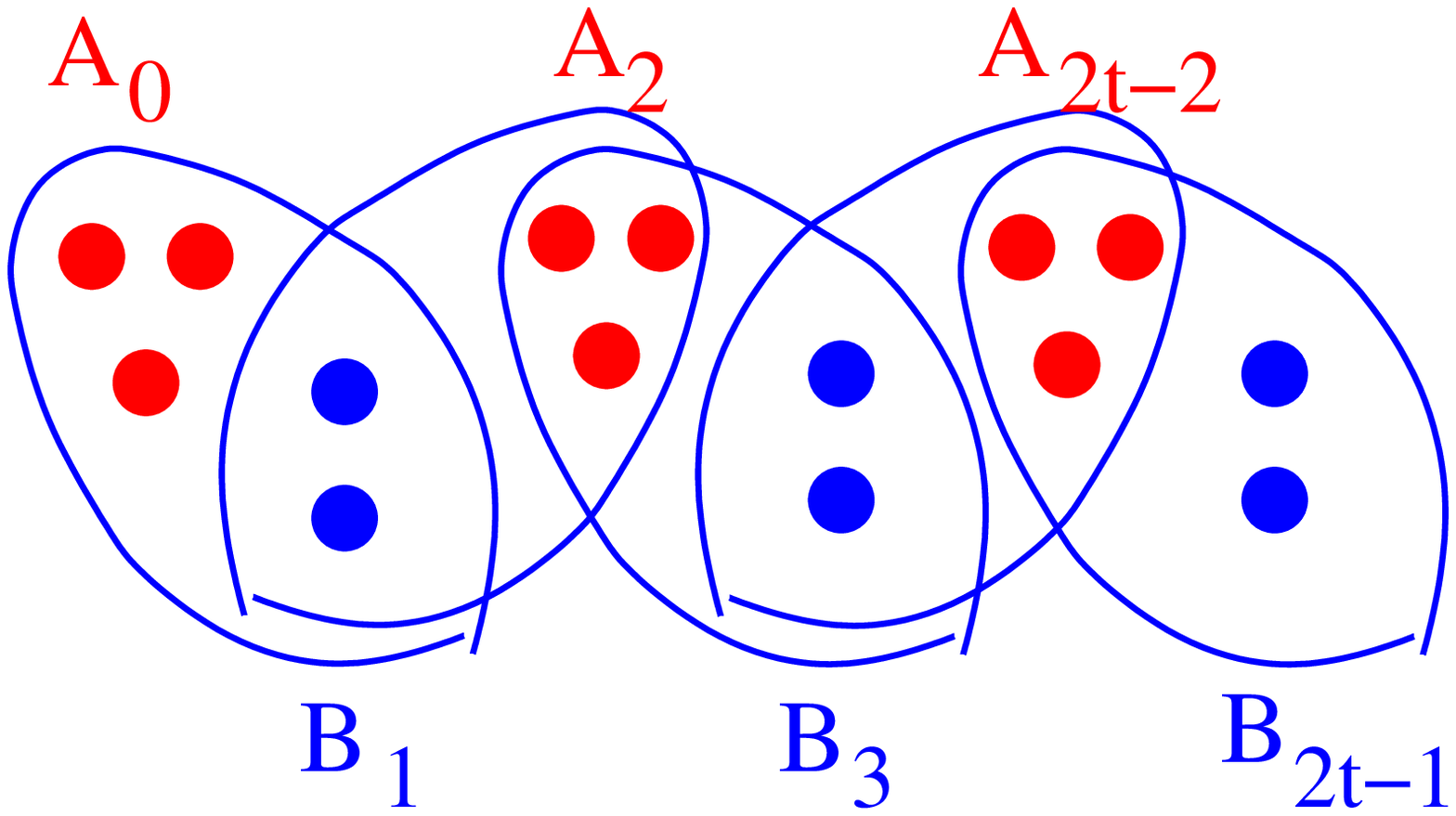} \quad\quad  \includegraphics[scale=0.3]{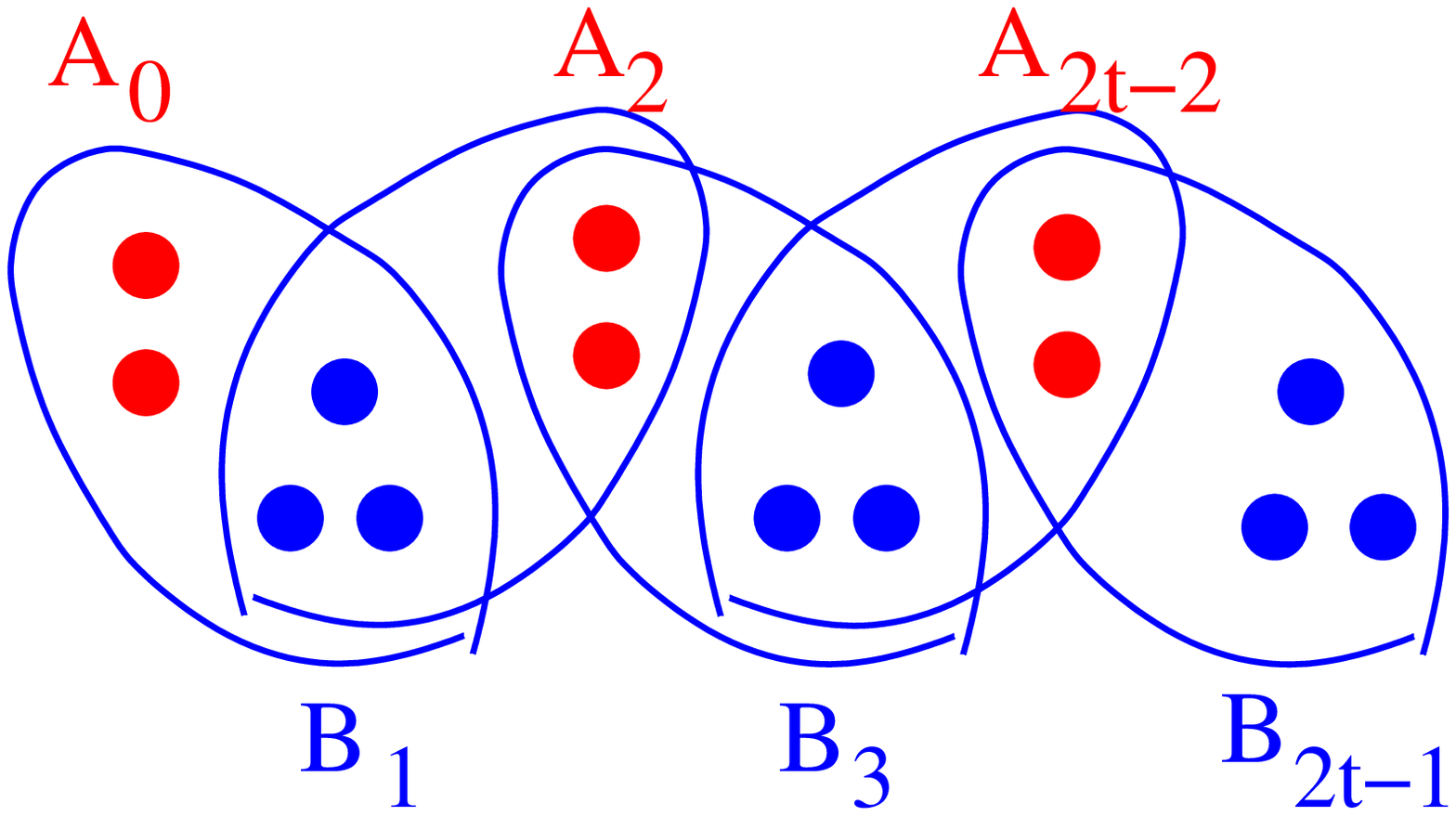} 
\quad\quad\quad  $P_5(3,2)=P_5(2,3)$.\end{center}

While $P_{2k-1}(a,b)=P_{2k-1}(b,a)$ we have that $P_{2k}(a,b)\neq P_{2k}(b,a)$\enskip  for $a\neq b$.

\begin{center}
    \includegraphics[scale=0.3]{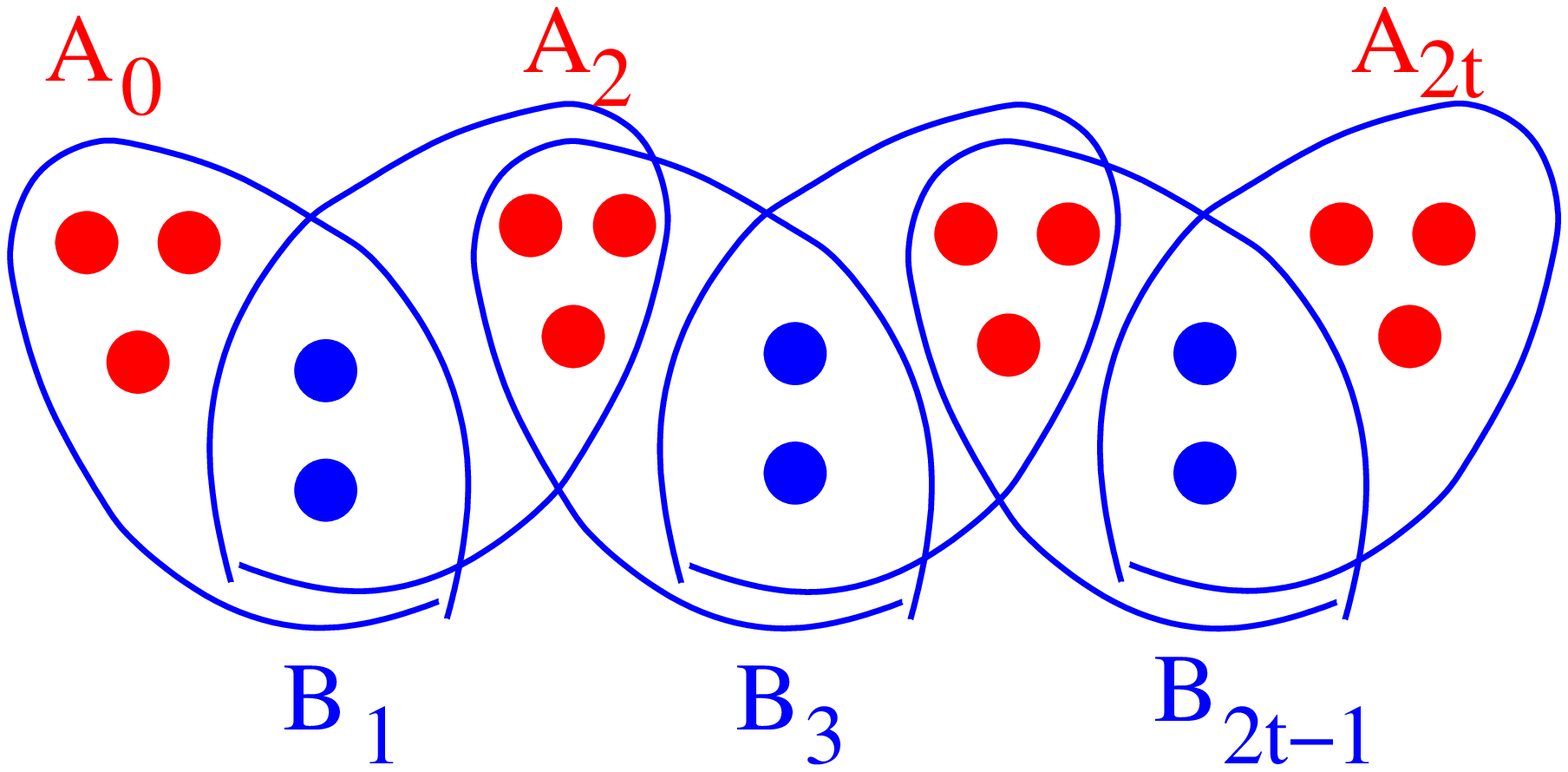} \quad\quad \includegraphics[scale=0.3]{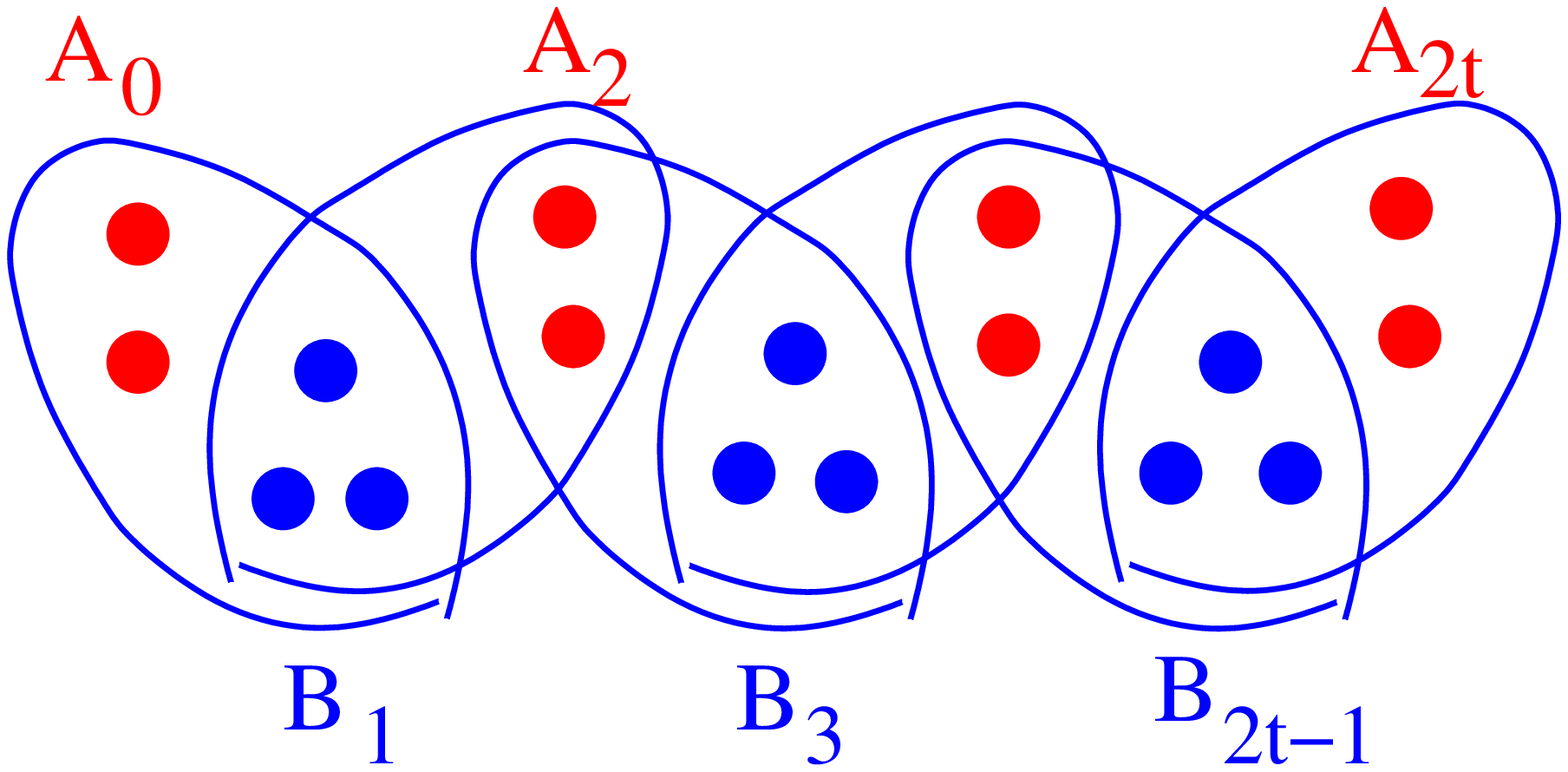} 
\end{center}

 ${}$ \vskip -0.8cm

{\bf On $(a,b)$-paths of length $3$.} \\
In the case $\ell=3$ an $(a,b)$-path has three $r$-sets, two of them are disjoint and they cover the third in a prescribed way.
For given $1\leq a,b <r$, $r=a+b$ and for $n> n_2(r)$,  F\"uredi and \"Ozkahya~\cite{FurOzk} showed that
\begin{equation*} 
   \ex_r(n, P_3(a,b)) = \binom{n-1}{r-1}.
   \end{equation*}

\medskip
{\bf Longer paths.} \\
Our first goal is to prove a nontrivial extension of the
Erd\H{o}s--Gallai Theorem~\eqref{eq:2} for $r$-graphs.

There are several ways to define a hypergraph path $P$.
One of the most difficult cases appears to be the case when $P$ is a {\em tight path} of length
$\ell$, namely the $r$-graph $Tight\, P_\ell^r$ with edges
$\{1,2,\dots,r\},\{2,3,\dots,r+1\},\dots,\{\ell,\ell+1,\dots,\ell+r-1\}$.
 The best known results~\cite{FJKMV_Zigzag} for this special case are 
	\[
\frac{\ell-1}{r}{n \choose r-1} \leq \ex_r(n,Tight\, P_\ell^r)  \leq \left\{\begin{array}{ll}
	\frac{\ell-1}{2}{n \choose r - 1} & \mbox{ if }r\mbox{ is even,} \\
	\frac{1}{2}(\ell + \lfloor \frac{\ell-1}{r}\rfloor){n \choose r - 1} & \mbox{ if }r\mbox{ is odd,}
	\end{array}\right.\]
where the lower bound holds as long as certain designs exist.

Another possibility is the
$r$-uniform {\em loose path} (also called {\em linear path})  $Lin\, P_\ell^r$, which is obtained from
$P^2_\ell$ by enlarging each edge with a new set of $(r-2)$
vertices such that these new $(r-2)$-sets are pairwise
disjoint (so $|V(P^r_\ell)|= \ell(r-1)+1$). Recently, the authors~\cite{FJS, KMV}
determined $\ex_r(n, Lin\, P_\ell^r)$ exactly for large $n$,
extending a work of Frankl~\cite{Frankl1977} who solved the
case $\ell=2$ by answering a question of Erd\H os and S\'os~\cite{S}
(see~\cite{KMW} for a solution for all $n$ when $r=4$).

Here we consider the $(a,b)$-blowup of $P_\ell$.
Since the case $\ell=2$ behaves somewhat differently, see~\eqref{eq:3} and~\eqref{eq:4},
 we only discuss the case $\ell\geq 3$.

Suppose that  $a+b=r$, $a,b\geq 1$, $r\geq 3$ and suppose that $\ell\in \{ 2k-1, 2k\}$, $\ell\geq 4$.
Furthermore, suppose that these values are fixed and $n\to \infty$ or $n> n_3(r,k)$.
Recall that $ \Psi_{t-1}(n,r):=\{ E\subset [n] : |E|=r,  E  \cap [k-1] \not= \emptyset\}$.
We have the lower bound 
\begin{eqnarray*}
  \ex_r(n,P_{2k}(a,b)) &\geq & \ex_r(n,P_{2k-1}(a,b))  \\  &\geq& |\Psi_{k-1}(n,r)|
=\binom{n}{r}-\binom{n-k+1}{r}  =(k - 1){n \choose r - 1}+ o(n^{r- 1})  .
        \end{eqnarray*}

Our main result (Theorem~\ref{mainexact}) implies that here equality holds
  for at least 75\% of the cases.
\begin{theorem} \label{th:path}
Suppose that  $a+b=r$, $a,b\geq 1$, $k\geq 2$.
Concerning the Turan number of $P_{\ell}(a,b)$,  the $(a,b)$ blowup of a path of length $\ell$, we have
\begin{gather*}
\ex_r(n,P_{2k-1}(a,b)) = { (k - 1){n \choose r - 1}+ o(n^{r - 1}) }\quad {\rm for \, \,  all \,\, odd }\,\,\, r,
    \\
{\ex_r(n,P_{2k}(a,b)) } = (k - 1){n \choose r - 1}+ o(n^{r - 1}) \quad {\rm for }\,\,\,  a>b.
   \end{gather*}
Moreover, if $a\neq b$, $a,b\geq 2$, $\ell=2k-1$, then  $\Psi_{k-1}(n,r)$ is the only extremal family.
\end{theorem}

The remaining cases ($\ell$ is even and $a\leq b$) are still open.

\begin{conjecture} 
$\Psi_{k-1}(n,r)$ gives the correct asymptotic of the Tur\'an number in all the above cases.
   \end{conjecture}

\section{Trees blown up, our main results}

Generalizing the Erd\H{o}s--Gallai Theorem~\eqref{eq:2},
 Ajtai, Koml\'os, Simonovits and Szemer\'{e}di~\cite{AKSS} claimed a proof of the
Erd\H{o}s--S\'{o}s Conjecture~\cite{ErdosSos}, showing
that if $T$ is any tree with $\ell$ edges, where $\ell$ is
large enough, then for all $n$,
\[ \ex_2(n,T) \leq \frac{1}{2}(\ell - 1)n.\]
 A more general conjecture
due to Kalai (see in~\cite{FF})  is about the extremal number for hypergraph
trees. A hypergraph $T$ is a {\em forest} if it consists of edges
$e_1,e_2,\dots,e_\ell$  ordered so that  for every $1<i\leq \ell$, there is $1\leq i'<i$ such that
$e_i \cap
(\bigcup_{j < i} e_j) \subseteq e_{i'}$.
A connected forest is called a {\em tree}.
If $T$ is $r$-uniform and for each $i>1$, $|e_i\cap (\bigcup_{j<i} e_j)|=r-1$,
then we say that $T$ is a {\em tight tree}.

\begin{conjecture} {\bf (Kalai)}
Let $T$ be an $r$-uniform tight tree with $\ell$ edges.
Then
\begin{equation*} 
  \ex_r(n,T) \leq \frac{\ell - 1}{r}{n \choose r - 1}.
 \end{equation*}
\end{conjecture}

When $r = 2$, this is precisely the Erd\H{o}s--S\'{o}s Conjecture.
A simple greedy argument shows that
\begin{proposition}\label{general-bound}
If $T$ is an $r$-uniform tight tree with $\ell$ edges and $G$ is an $r$-graph
on $[n]$ not containing $T$, then $|G|\leq (\ell-1)|\partial(G)|$.
\end{proposition}
Here $\partial(G)$ is the
family of $(r-1)$-sets that lie in some edge of $G$. We obtain 
\begin{equation*} 
  \ex_r(n,T) \leq (\ell - 1){n \choose r - 1}.
 \end{equation*}

Our goal is to prove a nontrivial extension of the Erd\H{o}s--Gallai Theorem and the Erd\H{o}s--S\'{o}s Conjecture
for $r$-graphs. 
To define the hypergraph trees we study in this paper, we make the following more general definition:

\begin{definition}
Let $s,t, a,b > 0$ be integers, 
 $r=a + b$, and let $H = H(U,V)$ denote a bipartite graph with parts $U = \{u_1,u_2,\dots,u_s\}$
and $V = \{v_1,v_2,\dots,v_t\}$.
Let $U_1, \dots, ,U_s$ and $V_1, \dots, V_t$ be pairwise disjoint sets, such that $|U_i|=a$  and $|V_j|=b$ for all $i,j$.
So $\left| \bigcup U_i\cup V_j \right| = as+bt$.

The {\em  $(a,b)$-blowup} of $H$, denoted by $H(a,b)$, is the
$r$-uniform hypergraph  with edge set
\[ H(a,b):= \{U_i \cup V_j :
u_iv_j \in E(H)\}\] 
\end{definition}

\vskip -0.2cm
    \includegraphics[height=0.9in]{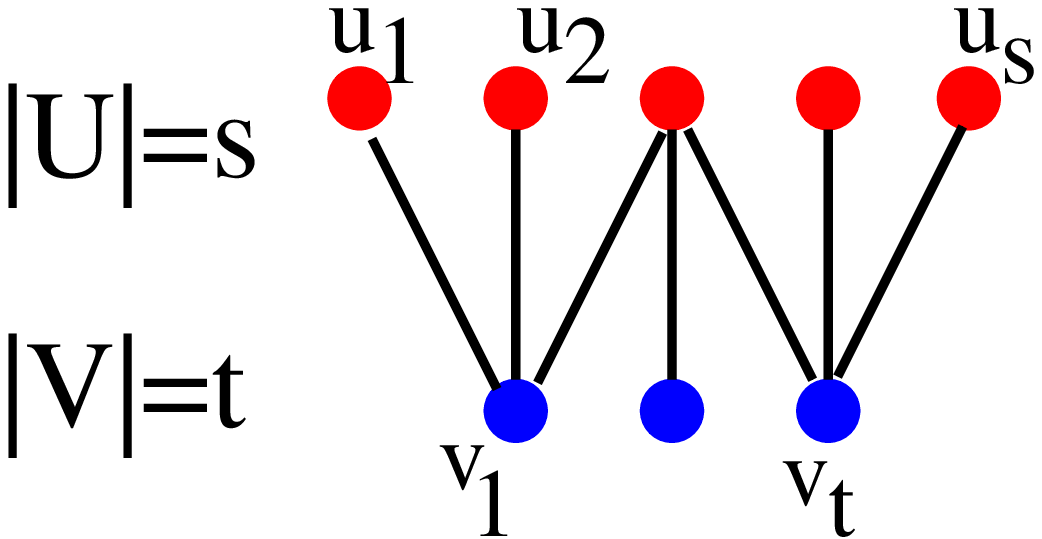} $\quad \Longrightarrow\quad $
    \includegraphics[height=1.2in]{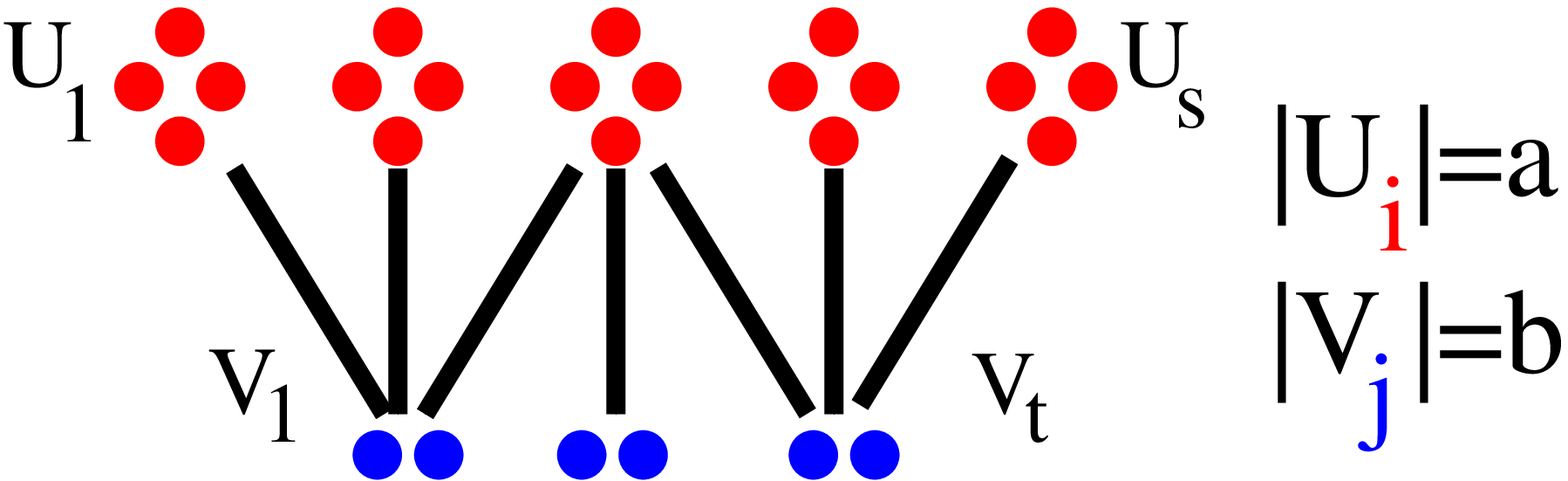}

  Since an $(a,b)$-blowup of a bipartite graph $H$ 
 $\sigma(H)$ is well defined.
Since deleting a vertex cover from a bipartite graph leaves an independent set,
each cross cut in a connected bipartite graph is one of its parts, 
 $\sigma(\mathcal{H})=\min\{s,t\}$.
Then the crosscut construction~\eqref{eq:psi1},  $ \Psi^1_{\sigma-1}(n,r):=\{ E\subset [n] : |E|=r,  |E  \cap [\sigma-1] |= 1\}$, yields that
\begin{equation}\label{eq:cclower}  (\sigma -1){n \choose r - 1} + o(n^{r - 1})
     = (\sigma-1) \binom{n-\sigma +1}{r-1} =  |\Psi^1_{\sigma -1}(n,r)|\leq  \ex_r(n,H).
 \end{equation}

 Let $\mathcal{T}_{s,t}$ denote the family of
  trees $T$ with parts $U$ and $V$ where $|U|= s$ and $|V| = t$.
We frequently say that $T$ is a tree with $s+t$ vertices.
Let $\mathcal{T}_{s,t}(a,b)$ denote the family of  $(a,b)$-blowups of
  trees $T\in \mathcal{T}_{s,t}$.
We frequently suppose that $a\geq b$ (but not always).

We investigate the problem of determining when crosscut
constructions are asymptotically extremal for
$(a,b)$-blowups of trees.
For other instances of hypergraph trees for which the crosscut constructions are asymptotically extremal,
 see~\cite{KMV2}.
Our main result is the following theorem.

\begin{theorem}\label{th:main}
Suppose  $r \geq 3$, $s,t \geq 2$, $a + b = r$, $b < a < r$.
\Tab
Then (as $n\to \infty$) any
$\mathcal{T}$-free $n$-vertex $r$-graph $H$ satisfies
\[ |H| \leq (t - 1){n \choose r-1} + o(n^{r-1}).\]
This is asymptotically sharp whenever $t\leq s$.
\end{theorem}
Indeed, in the case $t\leq s$ we have $\sigma(\mathcal{T})=t$ and~\eqref{eq:cclower} provides 
 a matching lower bound.

 A  vertex $x$ of  $T\in \mathcal{T}_{s,t}$ is called a {\em critical leaf} if $\sigma(T\setminus x)< \sigma(T)$.
In case of $t\leq s$ it simply means that $\deg_T(x)=1$ and $x\in V$.
(Similarly, a {\em critical leaf} of $\mathcal{T}=T(a,b) \in \mathcal{T}_{s,t}(a,b)$ with $t\leq s$ is a $b$-set $V_j$ in the part of size $t$ whose degree in $\mathcal{T}$ is one).
If such a vertex exists then we have a more precise upper bound.

\begin{theorem} \label{mainexact}
Suppose  $r \geq 5$, $2\leq t \leq s$, $a + b = r$, $b < a < r-1$.
\Tab 
Suppose that $T$ has a critical leaf.
Then for large enough $n$ ($n > n_0(T)$)
\[ \ex(n,\mathcal{T}) \le {n \choose r} - {n - t + 1
\choose r}.\]
If, in addition, $\tau(\mathcal{T})=t$,  then  equality holds above and the only example achieving the bound is  $\Psi_{t-1}(n,r)$.
\end{theorem}

Since $\tau(\Psi_{t-1}(n,r))=t-1$, no  $r$-graph $F$ with $\tau(F)\geq t$ is contained in $ \Psi_{t-1}(n,r)$.

\section{Asymptotics}

In this section we prove the asymptotic version of our main results, i.e., Theorem~\ref{th:main}.

\subsection{Definition of templates and a lemma.}

Throughout this section, $\mathcal{T} \in \mathcal{T}_{s,t}(a,b)$
and we suppose $\mathcal{T}$ is an $(a,b)$-blowup of a tree
$T$. If $H$ is an $r$-graph, then an {\em
$(a,b)$-template in $H$} is a pair $(A,B)$ where $A$ is
an $a$-uniform hypergraph on $V(H)$, $B$ is a $b$-uniform
matching on $V(H)$, and $V(A) \cap V(B) = \emptyset$.
Define the bipartite graph
\[ H_0 = H_0(A,B) = \{(e,f) \in A \times B : e \cup f \in H\}\]
and let $H_1 = H_1(A,B) = \{e \cup f : (e,f) \in H_0\}
\subset H$. By construction,  $|H_0| = |H_1|$.
We claim that {\em if $A$ and $B$ are
both matchings and $H_1(A,B)$ is $\mathcal{T}$-free, then}
\begin{equation}\label{template1}
|H_1(A,B)| \leq (t - 1)|A| + (s - 1)|B|.
\end{equation}
Indeed, otherwise
$|H_0(A,B)|=|H_1(A,B)| > (t - 1)|A| + (s - 1)|B|$
and $H_0$ has a minimum induced subgraph $H'_0(A',B')$
satisfying $|H'_0(A',B')|> (t - 1)|A'| + (s - 1)|B'|$. By minimality,
 $H'_0$ has
 minimum degree at least $t$ in $A'$ and
minimum degree at least $s$ in $B'$. This is
 sufficient to greedily construct a copy of $T$ in
$H'_0$. Since $H_1$ is an $(a,b)$-blowup of $H_0\supseteq H'_0$, this
shows $\mathcal{T} \subset H_1$.

We now prove a version
of (\ref{template1}) for templates, i.e., in the case when $A$
may  be not a matching:

\begin{lemma}\label{eglift}
Let $\delta > 0$ and let $\mathcal{T} \in
\mathcal{T}_{s,t}(a,b)$. Let $H$ be a $\mathcal{T}$-free
$r$-graph containing an $(a,b)$-template $(A,B)$. If $B =
B^0 \sqcup B^1$ and $d_H(e) \leq \delta n^b$ for every
$a$-set $e \subset V(H) \backslash V(B^1)$, then
\begin{equation}\label{template2}
|H_1(A,B)| \leq (t - 1) |A| + as n^{a - 1} (\delta|B^0| + |B^1|).
\end{equation}
\end{lemma}

\vspace{-0.2in}

\begin{proof} Let $\beta_0 =  a s \delta n^{a - 1}$ and $\beta_1 = as n^{a - 1}$. Let $H_1 = H_1(A,B)$ and $H_0 = H_0(A,B)$ and suppose
$|H_1| \geq (t - 1)|A| + \beta_0 |B^0| + \beta_1 |B^1|$.
By deleting vertices of $H_0$, we may assume
\begin{equation}\label{1124}
\mbox{\em
 $d_{H_0}(e)
\geq t$ for all $e \in A$ and for $i \in \{0,1\}$,
$d_{H_0}(e) > \beta_i$ for all $e \in B^i$. }
\end{equation}
Suppose
$\mathcal{T}$ is a blowup of a tree $T$, where $T$ has a unique
bipartition $(U,V)$ with $|U|=s, |V|=t$. We call an embedding of the $(a,b)$-blowup of a subtree
$T'$ of $T$ in $H_1(A,B)$ a {\em feasible embedding} if the $a$-sets corresponding to vertices
in $U$ are mapped to members of $A$ and the $b$-sets corresponding to vertices in $V$ are mapped to members of $B$.
It suffices to prove that any feasible embedding  $h$ of the $(a,b)$-blowup of any proper subtree $T'$ of $T$
can be extended to a feasible embedding $h'$ of the $(a,b)$-blowup of a subtree of $T$ that strictly contains $T'$.

Let $T'$ be given. Then there exists an edge $xy$ in $T$ with $x\in V(T')$ and $y\notin V(T')$.
Let $h$ be a feasible embedding of the $(a,b)$-blowup $\mathcal{T'}$ of $T'$ in $H_1(A,B)$.
First suppose that $x\in U$. Let $e$ denote the image under $h$ of $a$-set in $\mathcal{T'}$ that corresponds to $x$.
By our assumption $e\in A$. Hence by our earlier assumption, $d_{H_0}(e)\geq t$. Thus
$|\Gamma_{H_1}(e)|\geq t$. Since $\Gamma_{H_1}(e)\subseteq B$ is a matching of size at least $t$
and the $b$-sets corresponding to $V -  \{y\}$ are mapped to at most $t-1$ members of $B$,
there exists $f\in B$ such that $f\cap V(h(\mathcal{T'}))=\emptyset$. We can extend $h$ to a feasible embedding
of $T'\cup xy$ by mapping the $b$-set in $\mathcal{T}$ corresponding to $y$ to $f$.

Next, suppose $x\in V$. Let $e$ denote the image under $h$ of the $b$-set in $\mathcal{T'}$ that corresponds to $x$.
If there exists $f\in \Gamma_{H_1}(e)-V(h(\mathcal{T'}))$, then $h(\mathcal{T'})\cup \{e\cup f\}$ is a feasible embedding
of $T'\cup xy$. Hence we may assume that no such $f$ exists. If $e \in B^0$, then
we estimate  $d_{H_0}(e)$ by adding $a-b$ new vertices, one from $V(h(\mathcal{T'}))$ and all outside $V(B^1)$.
This yields
\[ d_{H_0}(e) \leq |V(h(\mathcal{T'})) \cap V(A)| \cdot n^{a - b - 1} \cdot \delta n^{b} \leq as \delta n^{a - 1} = \beta_0,\]
a contradiction to~\eqref{1124}. Note it is crucial here that $b < a$.
Similarly, if $e \in B^1$, then
\[ d_{H_0}(e) \leq |V(h(\mathcal{T'})) \cap V(A)| \cdot n^{a - 1} \leq as n^{a - 1} = \beta_1.\]
This contradicts $d_{H_0}(e) > \beta_1$ for $e \in B^1$. Hence we have shown that each feasible embedding of
$\mathcal{T'}$ can be extended. This completes the proof.
\end{proof}

\subsection{Proof of Theorem \ref{th:main}.}
In a few places of the proof we will use the following elementary fact or a slight variant of it.
Let $e$ be a fixed edge in $\binom{[n]}{p}$ and $H$ a $p$-graph on at most $n$ vertices.
Let $L$ be a copy of $H$ in $\binom{[n]}{p}$ chosen uniformly at random among all copies of $H$.
Then $P(e\in L)=|H|/\binom{n}{p}$.

Let $m$ be an integer satisfying $m > r^r$ and $m =
o(\sqrt{n})$. Let $f(m) = m^{-1/r} n^{r-1} + m^2 n^{r -
2}$. We show that if $H$ is $\mathcal{T}$-free for some
$\mathcal{T} \in \mathcal{T}_{s,t}(a,b)$, then
\[ |H| \leq (t - 1){n \choose r - 1} + O(f(m)).\]
In particular, taking $m = n^{1/3}$, we obtain
\[ |H| \leq (t - 1){n \choose r - 1} + O(n^{r - 1 -
1/(3r)}).\]

 \medskip
In our arguments below, for convenience, we assume $b$ divides $n$, since assuming so has
no effect on the asymptotic bound we want to establish.
Let  $D = \{e \in {V(H) \choose a} : d_H(e) \geq
n^{b}/m\}$ and $L$ be a smallest vertex cover of $D$,
meaning that every set in $D$ intersects $L$.
We claim
\begin{equation}\label{L=O}
|L| = O(m).
\end{equation}
Indeed, if $|L| \geq as m$, then $D$
has a matching $M$ of size $s m$. Each set in $M$ forms
an edge of $H$ with at least $n^b/m$ different $b$-sets,
and at most $a|M|n^{b-1}=asmn^{b-1}$ of these $b$-sets
intersect $V(M)$. By averaging, there is a matching $N$
of $b$-sets disjoint from $V(M)$ such that
\[ |H_0(M,N)| \geq \frac{|M|(n^b/m - asmn^{b-1})}{{n-1 \choose b-1}} > |M| \cdot \frac{n}{m} - |M| \cdot as m.\]
Since $n$ is large and $m = o(\sqrt n)$,  this is at
least
\[ (t - 1)|M| + \Bigl(\frac{n}{m} - t + 1 - asm\Bigr)|M|
\geq (t - 1)|M| + (s - 1)n > (t - 1)|M| + (s - 1)|N|.\]
By (\ref{template1}), we conclude that $\mathcal{T}
\subset H_1(M,N) \subset H$, a contradiction. This proves~\eqref{L=O}.

\medskip

Let $G = \{e \in H : |e \cap L| \leq 1\}$, so that
\begin{equation}\label{fbound}
|G| \geq |H| - |L|^2 n^{r-2} \geq |H| - O(m^2 n^{r - 2}).
\end{equation}
Let $R \subset V(G) \backslash L$ be a set whose elements
are chosen independently with probability $\alpha =
m^{-1/r}$, and $A = {R \choose a}$. Let $P$ be a random
partition of $V(G)$ into $b$-sets. Let $B$ denote the set
of $b$-sets in $P$ that are disjoint from $R$, and let
$H_1 = H_1(A,B)$. If $B^0 = \{e \in B : e \cap L =
\emptyset\}$ and $B^1 = \{e \in B : |e \cap L| \geq 1\}$,
then by (\ref{template2})  with $\delta = 1/m$, and using
$|B^1| \leq |L|$,
\[ |H_1| \leq (t - 1)|A| + O(n^{a-1}|B^0|/m) + O(n^{a-1}|L|).\]
Taking expectations over all choices of $R$ and $P$ and
using~\eqref{L=O} and $|B^0| \leq n$, we get
\begin{equation}\label{eq1}
E(|H_1|) \leq (t - 1)\alpha^a {n \choose a} + O(n^a/m).
\end{equation}
For $i \in \{0,1\}$, let $G_i = \{e \in G : |e \cap L| =
i\}$ and note $G = G_0 \cup G_1$. We observe that for an
edge $e \in G_0$,
\[ P(e \in H_1) = \frac{{r \choose b} \alpha^a (1 -
\alpha)^b}{{n - 1 \choose b - 1}} := p_0\] and for an
edge $e \in G_1$,
\[ P(e \in H_1) = \frac{{r - 1 \choose b - 1} \alpha^a (1
- \alpha)^{b - 1}}{{n - 1 \choose b - 1}} := p_1.\]
Since
$\alpha = m^{-1/r} < 1/r$ and $b \le (r-1)/2$,
$$p_0 = \frac{r}{b}(1 - \alpha)p_1 >  2p_1.
      $$ 
Therefore
\begin{eqnarray}
E(|H_1|)\geq p_0|G_0| + p_1|G_1| &=& (p_0 - p_1)|G_0| + p_1|G| > p_1|G|=
\frac{\alpha^a (r - 1)!(1-\alpha)^{b-1}}{a!n^{b - 1}}|G|.\label{upper2}
\end{eqnarray}
Combining this with (\ref{eq1}), using $(1-\alpha)^{-b+1}
= 1 - O(m^{-1/r})$ and after some simplification, we find
\begin{eqnarray*}
|G| &\leq& (t - 1){n \choose r - 1} + O(\alpha n^{r -
1}) + O(n^{r - 1}/\alpha^a m)  \\
&\leq& (t - 1){n \choose r - 1} + O(m^{-1/r}n^{r - 1}).
\end{eqnarray*}
Together with (\ref{fbound}), this gives the required
bound on $|H|$.  \qed

In fact, the proof of  Theorem \ref{th:main}
yields more then the theorem claims. We have the following fact.

\begin{corollary}\label{mainprime}
Let $0<\gamma< 1/t$, $b < a < r$, $a+b=r$, $t \leq s$. Let $n$ be
sufficiently large,  $r^r < m \le n^{\gamma}$ and $f(m) = m^{-1/r} n^{r-1} + m^2 n^{r - 2}$.
Let
$\mathcal{T} \in
\mathcal{T}_{s,t}(a,b)$  and  $H$ be
an $n$-vertex $\mathcal{T}$-free $r$-graph. If
\begin{equation}\label{asymptotic}
|H| = (t - 1){n \choose r - 1} + O(f(m))
\end{equation}
then some $F \subset H$ with $|F| = |H| - O(f(m))$ has a crosscut $L$ of size $O(m)$.
\end{corollary}

\begin{proof}
If $|H| = (t - 1){n \choose r - 1} + O(f(m))$, then the
upper and lower bounds for $E(|H_1|)$ given by
(\ref{eq1}) and (\ref{upper2}) differ by $O(n^a/m)$. By (\ref{upper2}) they also differ by at
least $(p_0 - p_1)|G_0|$ so
\[  (p_0 - p_1)|G_0| = O(n^a/m).\]
Using $p_0 > (1 + 1/r)p_1$, we get $p_1|G_0| = O(n^a/m) $ and this shows $|G_0| = O(f(m))$. Setting $F = G_1$,
$L$ is a crosscut of $F$ and $|F| = |H| - O(f(m))$.
\end{proof}

\section{Stability}

The aim of this section is to prove the following
stability theorem. It is important throughout this
section that $t \leq s$, so that for $\mathcal{T} \in
\mathcal{T}_{s,t}(a,b)$, we have $\sigma(\mathcal{T}) =
t$ and therefore $\Psi^1_{t-1}(n,r)$ does not contain
$\mathcal{T}$. The following theorem says that if $H$ is
a $\mathcal{T}$-free $r$-graph on $n$ vertices and $|H|
\sim |\Psi_{t-1}(n,r)|$, then $H$ is obtained by adding or
deleting $o(n^{r - 1})$ edges from $\Psi_{t-1}(n,r)$.

\begin{theorem}\label{th8}
Let $\mathcal{T} \in \mathcal{T}_{s,t}(a,b)$, where $b <
a < r - 1$, $t \leq s$. 
Let $H$ be a
$\mathcal{T}$-free $n$-vertex $r$-graph with $|H| \sim (t
- 1){n \choose r - 1}$. If $\mathcal{T}$ has a critical leaf,
then there exists a set $S$ of $t
- 1$ vertices of $H$ such that $|H - S| = o(n^{r - 1})$.
\end{theorem}

\subsection{Degrees of sets.} By Corollary~\ref{mainprime} with $r^r < m = o(n^{1/(t+1)})$ there exists
$F \subset H$ such that $|F| \sim |H|$ and $F$ has a
crosscut $L$ of size $O(m)$. Our first claim says that
most elements of $\partial F$ have degree $t - 1$ in $F$.

\medskip

{\bf Claim 1.} {\em There are ${n \choose r-1} -
o(n^{r-1})$ sets $e \in \partial F - L$ such that $d_F(e)
= t - 1$.}

{\bf Proof.} Suppose $\ell$ sets $e \in \partial F - L$
have $d_F(e) \geq t$. By the definition of $L$, $\Gamma(e)\subseteq L$ for each $e \in \partial F - L$.
Let $Z$ be a crosscut of
$\mathcal{T}$ with $|Z| = t$ contained in $B$ and let $\mathcal{T}^*
= \{e \backslash Z : e \in \mathcal{T}\}$. Then
$\mathcal{T}^*$ is an $(a,b-1)$-blowup of $T$.
Proposition~\ref{general-bound} implies 
$$\ex(n, \mathcal{T}^*) < (s + t)n^{r-2}.$$

By the pigeonhole principle, there exists a set $S
\subset L$ with $|S| = t$ such that at least $k =
\ell/|L|^{t}$ sets $e \in \partial F - L$ have
$\Gamma_F(e) \supseteq S$. If $k > \ex(n,\mathcal{T}^*)$,
then $\mathcal{T}^* \subset \partial F - L$ and for all
$e \in \mathcal{T}^*$, $\Gamma_G(e) \supseteq S$.
Now we can lift $\mathcal{T}^*$  to $\mathcal{T}
\subset F$ via $S$. Indeed, we can greedily enlarge each of the
$(b - 1)$-sets that form $\mathcal{T}^*$ to a $b$-set
by adding an element of $S$. This contradicts the choice of $H$.
We therefore suppose that
\[ \ell/|L|^{t} = k \leq \ex(n,\mathcal{T}^*) \leq (s+t)n^{r-2}\]
which gives $\ell \leq (s+t)|L|^{t}n^{r-2} = O(n^{r -
2}m^t)$. As $|F| \sim |H| \sim (t - 1){n \choose r - 1}$,
and the number of $(r-1)$-sets in $V(F)-L$ is at most ${n
\choose r-1}$, the average degree of sets in $\partial F
- L$ is at least $t - 1 - o(1)$. We have already argued
that at most $O(n^{r -2}m^t)$ of these sets have degree larger than $t -
1$. Furthermore, none of them has degree greater than $m$.
Hence the number of sets in $\partial F
- L$ of degree at most $t-2$ is $z$, then we have inequality
$$(t-1){n \choose r - 1}-x+m O(n^{r -2}m^t)\geq (t-1){n \choose r - 1}(1-o(1)).$$
Since $m\ n^{r -2}m^t=o(n^{r-1})$,
  we conclude that $x=o\left({n \choose r - 1}\right)$. This yields the claim. \qed

\medskip

\subsection{Proof of Theorem~\ref{th8}}

Let $S_1,S_2,\dots,S_k$ be an enumeration of the $(t -
1)$-element subsets of $L$, and let $F_i$ denote the
family of $(r - 1)$-element sets $e$ in $V(F) \backslash
L$ such that $\Gamma_F(e) = S_i$. By   Claim~1, $|F_1 \cup F_2
\cup \dots \cup F_k| \sim {n - |L| \choose r - 1}$.
Suppose $k \geq 2$. By definition, for $i \neq j$, $F_i \cap F_j =
\emptyset$.
 Therefore,
\[ \sum_{i = 1}^k |F_i| \sim {n \choose r - 1}.\]
For each $i \in [k]$, if $|F_i| = o(n^{r - 1}/k)$, let
$G_i$ be an empty $(r - 1)$-graph, if $|F_i| =
\Omega(n^{r-1}/k)$, then delete edges of $F_i$ containing $a$-sets or $b$-sets of "small" degree until we
obtain either an empty $(r - 1)$-graph or an $(r -
1)$-graph $G_i$ such that
\begin{equation}\label{1125}
\mbox{\em  $d_{G_i}(e) > r(s + t)n^{r - 2 - a}$ $\forall$ $a$-set $e \in\partial_a G_i$,
and $d_{G_i}(f)
> r(s + t)n^{r - 2 - b}$ $\forall$ $b$-set $f \in \partial_b G_i$.}
\end{equation}
By construction, $|G_i| \geq |F_i| -
2r(s + t)n^{r - 2}$ and since $F_i = \Omega(n^{r-1}/k)$
and $k \leq |L|^t \leq O(m^t) = o(n)$, whenever $G_i$ is
non-empty we have
\[ |G_i| = (1 - o(1))|F_i|.\]
We conclude that if $G = \bigcup G_i$ then $|G| = (1 -
o(1))|F| \sim {n \choose r - 1}$ and
\begin{equation} \label{Gi-sum}
\sum_{i = 1}^k |G_i| \sim {n \choose r - 1}.
\end{equation}

{\bf Claim 2.} {\em For $i \neq j$, $\partial_a G_i \cap
\partial_a G_j = \emptyset$.}

{\it Proof.} Let $W$ be a tree obtained from the tree $T$
by deleting a leaf vertex $x$ with unique neighbor $y \in
T$, such that $x$ is in the part of $T$ of size $t$.
Suppose some $a$-set $e$ is contained in $\partial_a G_i
\cap
\partial_a G_j$. By~\eqref{1125},  we can greedily grow $W(a,b-1)$
in $G_j$ such that $e$ is the blowup of $y$. By adding
one vertex of $S_j$ to each $b - 1$-set in $W(a,b-1)$, we
obtain $W(a,b)$. Now there exists $x' \in S_i \backslash
S_j$. Since $d_{G_i}(e) > r(s + t)n^{r - 2 - a}$, there
exists an edge $f \in G_i$ containing $e$, such that $f
\cap V(W(a,b-1)) = \emptyset$, and therefore $f \cup
\{x'\} \in F$ plus $W(a,b)$ gives the tree $T(a,b)$, with
$f \backslash e$ the blowup of $x$. This proves the
claim. $\Box$

\bigskip

Now we prove Theorem~\ref{th8}. Since $a\leq r-2$, by Claim~2, for all $i\neq j$,
$\partial_{r-2} G_i \cap \partial_{r-2} G_j=\emptyset$. Without loss of generality,
suppose that for some $0\leq p\leq k$,  $|G_1|\geq |G_2|\geq \ldots\geq |G_p|\geq 1$
and $G_i=\emptyset$ for $p+1\leq i\leq k$.
For each $i\in [p]$, let $y_i\geq r-1$ denote the real such that $|G_i|=\binom{y_i}{r-1}$.
Then $y_1\geq y_2\geq \cdots\geq y_p$.
By the Lov\'asz form of the Kruskal-Katona theorem, for each $i\in [p], |\partial_{r-2}(G_i)|\geq \binom{y_i}{r-2}$.
By the disjointness of the $\partial_{r-2}(G_i)$'s, we have
$$\sum_{i=1}^p \binom{y_i}{r-2}\leq \binom{n}{r-2}.$$
For each $i\in [p]$, since $\binom{y_i}{r-1}=\frac{y_i-r+2}{r-1}\binom{y_i}{r-2}\leq \frac{y_1-r+2}{r-1}\binom{y_i}{r-2}$,
by \eqref{Gi-sum} we have
$$(1-o(1))\binom{n}{r-1}\leq \sum_{i=1}^p |G_i|=\sum_{i=1}^p \binom{y_i}{r-1}\leq \frac{y_1-r+2}{r-1}\sum_{i=1}^p \binom{y_i}{r-2}\leq \frac{y_1-r+2}{r-1}\binom{n}{r-2}.$$
From this, we get $y_1\geq n-o(n)$. Hence $|F_1|\geq |G_1|=\binom{y_1}{r-1}\geq \binom{n}{r-1}-o(n^{r-1})$.
Hence there exists $S=S_1\subset L$ such that $(t-1)\binom{n}{r-1}-o(n^{r-1})$ edges of $F$ consists of one vertex in $S$ and $r-1$ vertices disjoint from $S$. $\Box$

\section{Exact results}

The aim of this section is to prove the following
theorem, which completes the proof of Theorem
\ref{mainexact}:

\begin{theorem}
Let $t \leq s$, $b < a < r - 1$ with $a + b = r$ and
$\mathcal{T} \in \mathcal{T}_{s,t}(a,b)$ such that $\mathcal{T}$ has a critical leaf and $\tau(\mathcal{T})=t$.
If  $n$ is large and $H$ is a
$\mathcal{T}$-free $n$-vertex $r$-graph with $|H| \geq {n
\choose r} - {n - t + 1 \choose r}$, then $H \cong \Psi_{t-1}(n,r)$.
\end{theorem}

To prove this, we aim to show that the set $S$ given by Theorem~\ref{th8}
 is a vertex cover of $H$. We prove the following
consequence of Claim~1:

\medskip

{\bf Claim~3.} {\em Let $\Delta_u = (t - 1){n-u \choose r -
1 - u}$. Then for each $\delta > 0$, there exists $G
\subset F$ with $|G| \sim |F|$ such that for any $u$-set
$e \subset V(G)$ with $u < r$ and $d_G(e)> 0$,} either
\begin{center}
\begin{tabular}{lp{5in}}
(i) & $|e \cap S| = 0$ and $d_G(e) \geq (1 - \delta)\Delta_u$ or \\
(ii) & $|e \cap S| = 1$ and $d_G(e) \geq r(s+t) n^{r - 1 - u}$.\\
\end{tabular}
\end{center}

\begin{proof}
Let $S$ be the $(t-1)$-set given by Theorem~\ref{th8} and
$K$ be the set of edges of $F$ containing some $e \in
\partial F - S$ with $d_F(e) = t - 1$.  By Claim~1, $|K|
\sim |F|$. Also, every $r$-set in $K$ has one
point in $S$ and $r - 1$ points in $V(K) \backslash S$.
Since $d_K(e) = t - 1$ for all $e \in \partial K - S$,
every $u$-set in $V(K) \backslash S$ has degree at most
$\Delta_u$ in $K$.

\medskip

We repeatedly delete edges from $K$ as follows. Suppose
at some stage of the deletion we have a hypergraph $K'$.
If there exists a $u$-set $e$ for some $u<r$ such that
\begin{center}
\begin{tabular}{lp{5in}}
(i') & $|e \cap S| = 0$ and $d_{K'}(e) < (1 - \delta)\Delta_u$ or  \\
(ii') & $|e \cap S| = 1$ and $d_{K'}(e) < r(s+t) n^{r - 1 - u}$\\
\end{tabular}
\end{center}
then delete all edges of $K'$ containing $e$. Let $G$ be
the hypergraph obtained at the end of this process. We
shall prove $|G| \sim |K|$. To this end, suppose that
$|G| = |K| - \eta(t - 1){n \choose r - 1}$, and we show
$\eta = o(1)$ to complete the proof. Consider two cases.

\medskip

{\bf Case 1.} {\em At least $\frac{\eta}{2}(t - 1){n
\choose r - 1}$ edges of $K$ were deleted due to (ii').}

\medskip

In this case, there exists $u < r$ such that the set $H'$
of edges of $K$ deleted due to (ii') on $u$-sets satisfies
$|H'| \geq \frac{\eta}{2r}(t - 1){n \choose r - 1}$. Then by (ii'), and since the number of $u$-sets
with one vertex in $S$ is $|S|{n - |S| \choose u - 1}$,
\[ |H'| \leq |S|{n - |S| \choose u - 1} \cdot r(s+t) n^{r - 1 - u}
< |S|r(s+t) n^{r - 2}.\] Since $|H'| \geq
\frac{\eta}{2r}{n \choose r - 1}$ and $|S| = t-1$, this
gives $\eta = o(1)$.

\medskip

{\bf Case 2.} {\em At least $\frac{\eta}{2}(t - 1){n
\choose r - 1}$ edges of $K$ were deleted due to (i').}

\medskip

In this case, there exists $u < r$ such that the set $H'$
of edges of $K$ deleted due to (i') on $u$-sets satisfies
$|H'| \geq \frac{\eta}{2r}(t - 1){n \choose r - 1}$. Let
$U_1$ be the set of $u$-sets in $V(K) \backslash S$ on
which edges of $K$ were deleted due to (i'), and let
$U_2$ be the remaining $u$-sets in $V(K) \backslash S$.
Then
\[ |U_1| > \frac{|H'|}{(1-\delta)\triangle_u} \geq \frac{\eta (t - 1){n \choose r - 1}}{2r(t - 1){n \choose r - 1 - u}}.\]
If $n$ is large enough, then this is at least
$\frac{\eta}{4r{r - 1 \choose u}}{n \choose u}$. Let
$\gamma = \frac{\eta}{4r{r - 1 \choose u}}$. Then
\begin{eqnarray*}
|K|{r - 1 \choose u} &=& \sum_{e \in {V(K)
\backslash L \choose u}} d_K(e) \\
&=& \sum_{e \in U_1} d_K(e)
+ \sum_{e \in U_2} d_K(e) \\
&\leq& |U_1| (1 - \delta)\Delta_u
+ |U_2|\Delta_u \\
&\leq& \gamma(1 - \delta) {n \choose u}\Delta_u + (1 - \gamma){n \choose u}\Delta_u \; = \; (1 - \gamma \delta){n \choose u}\Delta_u.
\end{eqnarray*}
Here we used $|U_1| + |U_2| \leq {n \choose u}$.
Therefore
\[ |K| \leq (1 - \gamma \delta) \frac{{n \choose u}
\Delta_u}{{r - 1 \choose u}} = (1 - \gamma \delta)(t
- 1){n \choose r - 1}.\] Since $|K| \sim |F| \sim (t
- 1){n \choose r - 1}$, $\gamma \delta = o(1)$. Since
$\delta > 0$ and $\gamma = \frac{\eta}{4r{r - 1 \choose
u}}$, this implies $\eta = o(1)$, as required.
\end{proof}

\medskip

Let $\mathcal{T} \in \mathcal{T}_{s,t}(a,b)$ have a critical
leaf with $\tau(\mathcal{T})=t \le s$, $a+b=r$, $b<a<r-1$,  and let $H$ be a $\mathcal{T}$-free $n$-vertex $r$-graph with
$|H| \geq {n \choose r} - {n - t + 1 \choose r}$. We aim to show that $S$ is a vertex cover of $H$, which
gives $H \cong \Psi_{t-1}(n,r)$, as required. To this end,  let
$H_i = \{e \in H : |e \cap S| = i\}$. So we have to show $H_0 = \emptyset$.

Since $\mathcal{T}$ has a critical leaf, there is a $b$-set $e'$ of $\mathcal{T}$ in the part of size $t$
with $d_{\mathcal{T}}(e')=1$. Let $\mathcal{T'}$ be the tree obtained from $\mathcal{T}$ by deleting the edge containing $e'$.  So  $V(\mathcal{T'})$
has one part comprising $t-1$ sets, each of size $b$ and the other part comprising $s$ sets, each  of size $a$. It has a crosscut of size $t-1$ by picking one vertex from each of the $b$-sets above.

Let ${\mathcal K}^1$ be the set of $r$-sets of $[n]$ that have exactly one vertex in $S$.
A subfamily $T \subset {\mathcal K}^1$ is a {\em potential tree} if
\newline\quad
1) $T \cong \mathcal{T}'$
\newline\quad
2) the $t-1$ vertices of $S$ play the role of the crosscut vertices of $\mathcal{T}'$ described above
\newline\quad
3) $e_0$ is an $a$-set in $V(T)$ with $e_0 \in \partial_aH_0$
\newline\quad
4) $e_0 \subset e \in H_0$
\newline\quad
5) $T \cup e$ is a copy of $\mathcal{T}$.

Fix an $a$-set $e_0 \in \partial_a H_0$ and suppose $e_0 \subset e \in H_0$. If $T \subset H_1$ is a potential tree as described above, then $T \cup \{e\}$ is a copy of $\mathcal{T}$ in $H$, a contradiction. So for each such
potential tree $T$, there exists $f \in T-H_1$. Let us call this a {\em missing edge}.
Let $m = as+bt-b$ be the number of vertices of each potential tree.
The number of potential trees containing a fixed missing edge $f$ is at most
\[  {n-|S|-(a+b-1) \choose m-|S|-(a+b-1)} \cdot c(\TT),\]   
where $c(\TT)$ is the number of ways we can put a potential tree using $f$ into the set $M$ with
$|M|=m$ and
  $(S\cup f)\subset M\subset [n]$,
(note that $|f\cap S|=1$).

On the other hand, each $e_0\in \partial_a H_0$  and a subset $M'$ with $|M'|=m$ and
  $S\subset M'\subset ([n]-e_0)$ carries at least one potential tree so  the total number of potential trees is at least 
\[ |\partial_a H_0|  {n - |S|-a\choose m-|S|-a}.\]

It follows that the number of missing edges is at least $c|\partial_a H_0| n^{b - 1}$ for some $c> 0$. Therefore
\[ |H| = |H_0| + |H_1| + |H_2| + \dots + |H_r| \leq {n \choose r} - {n - t + 1 \choose r} + |H_0| - c |\partial_a H_0| n^{b - 1}.\]
By Proposition~\ref{general-bound} and the fact that $\mathcal{T}$ is contained  in a tight tree on $V(\mathcal{T})$, $|H_0| < c'|\partial H_0|$
for some constant $c'$.

Next, we observe that $\partial H_0 \cap \partial
G = \emptyset$, for otherwise we can use Claim~3 to greedily build a copy of $\mathcal{T}$
using the edge of $H_0$, and whose remaining
edges form a copy of $\mathcal{T}'$ and come from $G$. In particular, since $|\partial G|
\sim {n \choose r - 1}$, $|\partial H_0| = o(n^{r - 1})$.
Writing $|\partial H_0| = {x \choose r - 1}$ for some
real $x$, we have $|\partial_a H_0| \geq
{x \choose a}$, by the Kruskal-Katona Theorem. Therefore
\[ |H_0| - c |\partial_a H_0| n^{b - 1} \leq c'|\partial H_0| -  c |\partial_a H_0| n^{b - 1}
\leq c'{x \choose r - 1} - c n^{b - 1} {x \choose a}.\]
Since $x = o(n)$, for large enough $n$ the above
expression is negative, unless $|\partial H_0| =
|\partial_a H_0| = 0$. We have shown that if $|H| \geq {n
\choose r} - {n - t+ 1 \choose r}$, then $H_0 =
\emptyset$ and $|H| = {n \choose r} - {n - t + 1
\choose r}$, as required. \qed

\section{Concluding remarks}\label{se:7}

In this paper we determined for $b \leq a < r$
the asymptotic behavior of $\ex_r(n,\mathcal{T})$ when
$\mathcal{T} \in \mathcal{T}_{s,t}(a,b)$ is an
$(a,b)$-blowup of a tree $T$ with parts of sizes $s$ and
$t$ where $s \geq t$ and $\sigma(\mathcal{T})=t$. The extremal problem appears to be
more difficult when $s < t$, in which case the smallest
crosscut of $\mathcal{T}$ has size $s$. We pose
Conjecture \ref{general}, which covers all cases except
$a = r - 1$.

\begin{conjecture}\label{general}
If $\mathcal{T} \in \mathcal{T}_{s,t}(a,b)$ where $b \leq
a < r - 1$, $\sigma = \sigma(\mathcal{T}) = \min\{s,t\}$,
and $H$ is a $\mathcal{T}$-free $n$-vertex $r$-graph,
then for large enough $n$, $|H| \leq (\sigma - 1){n
\choose r - 1} + o(n^{r - 1})$, with equality only if $H$
is isomorphic to a hypergraph obtained from
$\Psi_{\sigma -1}(n,r)$ by adding or deleting $o(n^{r -
1})$ edges.
\end{conjecture}

\medskip 
 {\bf The case $a=r-1$.} \quad
If $t > s$ (and $n\geq |V(\mathcal{T})|$),  then
$\Psi^1_{t-1}(n,r)$ contains $\mathcal{T}$ so Conjecture~\ref{general} does not hold.
Since $\Psi_{s-1}^1(n,r)$ does not contain $\mathcal{T}$, it is
natural to ask whether $\Psi^1_{s-1}(n,r)$ is (asymptotically) extremal for
$\mathcal{T}$.  In some cases when $a=r-1$, this is certainly not so because
certain Steiner systems do not contain a blowup of a star $K_{1,t}$ and are denser than $\Psi_{s-1}(n,r)$. More precisely:
\Tab
Suppose $a=r-1$ and let $\lambda= \max_{x\in U} \deg_T(x)$.
 Then $\ex(n,\mathcal{T})$ is at least the
number of edges in a Steiner $(n,r,r-1,\lambda-1)$-system -- an
$r$-graph on $n$ vertices where each $(r - 1)$-set is
contained in exactly $\lambda - 1$ edges.
In this case,
$\ex(n,T(r-1,1)) \geq \frac{\lambda-1}{r}{n \choose r - 1}$ for
infinitely many $n$ (due to the existence of those
designs~\cite{Keevash}) whereas $\sigma(T) = s$ and it could be much less than $ \frac{\lambda-1}{r}$. 

\medskip

{\bf No stability for  $a = r - 1$.}\quad
It is important in the above proof that $a
\neq r - 1$. If $a = r - 1$, then there is no
stability theorem: consider for instance an
$(r-1,1)$-blowup $\mathcal{T}$ of a path with four edges.
Let $H$ be the $n$-vertex $r$-graph constructed as
follows. Let $V(H) = [n]$, let $G_1 \sqcup G_2$ be a
partition of the edge set of the complete $(r - 1)$-graph on $\{ 3,4, \dots, n\}$, and let
 $H$ consist of the edges $e \cup \{i\}$ such that
$e \in G_i$, for $i \in \{1,2\}$. Then $|H| = {n - 2
\choose r-1}$ and $H$ does not contain $\mathcal{T}$.

{\bf The case  $a=b=r/2$.}\quad
Let $T$ be a tree on $s+t$ vertices then for  $\TT=T(r/2,r/2)$
one can use an argument of Frankl~\cite{Frankl} (applied by many others, see~\cite{MV})
to prove that
\begin{equation}\label{eq:r/2}
  \ex_r(n,\TT) \leq \frac{\ex(\lfloor 2n/r\rfloor, T)}{\binom{\lfloor 2n/r \rfloor}{2}} {n\choose r}
   \sim  \frac{\ex(\lfloor 2n/r\rfloor, T)}{\lfloor 2n/r \rfloor} {n\choose r-1} .
\end{equation}
Indeed, similarly to the idea of templates, given a $\TT$-free $r$-graph $H$ on $n$
vertices take a
random partition of $[n]$ into $r/2$-sets, (where for simplicity $r/2$ divides $n$),
and consider only those $r$-edges of $H$ which are unions of two partite sets.
Then this subfamily consists of at most $\ex(2n/r, H)$ edges of $H$, out of the possible $\binom{2n/r}{ 2}$.

The bound is asymptotically tight, due to $\Psi^1_{t-1}(n,r)$,
  if $\sigma(\TT)=t$ and $T$ has $2t-1$ edges.
So the inequality~\eqref{eq:r/2} completes the proof of Theorem~\ref{th:path}
 showing that $ \ex_r\left(n, P_{2k-1}\left(\frac{r}{2},\frac{r}{2}\right)\right) \sim (k-1)\binom{n}{r-1}$
 (the other cases follow from Theorems~\ref{th:main} and~\ref{mainexact}).
It also gives a better upper bound for the even length, $\ex_r\left(n, P_{2k}\left(\frac{r}{2},\frac{r}{2}\right)\right) \leq (1+o(1))\left(k-\frac{1}{2}\right)\binom{n}{r-1}$.

However, the proof of~\eqref{eq:r/2}  does not reveal the extremal structure.

{\bf The case of forests.}\quad
Many of our ideas can be generalized for the case of $\TT=F(a,b)$, when $F$ is a {\em forest},
 but we do not have a general conjecture. 

\begin{problem}\label{forest}
Given $a,b\geq 1$ and a forest $F$ on $s+t$ vertices.
Determine $\lim_{n\to \infty} \ex(n, F(a,b))\binom{n}{r-1}^{-1}$.
    \end{problem}

{\bf Other bipartite graphs.}\quad
The class of $(a,b)$-blowups of bipartite
graphs contains well-studied instances including blowups
of complete bipartite graphs. In particular,
F\"{u}redi~\cite{Furedi} made the following conjecture
for blowups of a 4-cycle. Let $\mathcal{C}_4^r=\{C_4(a,b): a+b=r, a,b>0\}$.

\begin{conjecture}[\cite{Furedi}] If $r \geq 3$ then
$ \ex(n,\mathcal{C}_4^r) \sim  \dbinom{n}{r - 1}$.
\end{conjecture}

The current record is due to
Pikhurko and the last author~\cite{PV}, who showed
\[ \ex_r(n,\mathcal{C}_4^r) \lesssim (1 + \frac{2}{\sqrt{r}}){n
\choose r - 1}\]
and $\ex_3(n,C_4(2,1)) \lesssim \frac{13}{9}{n \choose
2}$. When $G$ is an even cycle
of length six or more, it is only known~\cite{JL} that $\ex_r(n,G(a,b))=\Theta(n^{r-1})$ and
 the asymptotic behavior of $\ex_r(n,G(a,b))$ is not known. One can show, however, that
for $F = K_{s,t}(a,b)$ with $a + b = r$, $b \leq a$, and
$t$ sufficiently large as a function of $s$ and $r$,
\[ \ex_r(n,F) = \Theta(n^{r - \frac{1}{s}})\]
via a randomized algebraic construction.

\end{document}